\documentclass[11pt]{amsart}
\usepackage{geometry}   
\usepackage{setspace}
\usepackage{graphicx}
\usepackage{tikz}
\usepackage[pagewise]{lineno}
\usetikzlibrary{arrows,decorations.markings, decorations.pathreplacing, calc}
\usepackage{amssymb, amsthm}
\usepackage{epstopdf}
\newcommand{\midarrow}{\draw[postaction={decorate}]}

 \newtheorem{theorem}{Theorem}[section]
    \newtheorem{corollary}[theorem]{Corollary}
   \newtheorem{lemma}[theorem]{Lemma}
    \newtheorem{proposition}[theorem]{Proposition}  
       
    \theoremstyle{definition}
\newtheorem{definition}[theorem]{Definition}
\newtheorem{remark}[theorem]{Remark}

 \newcommand{\Z}{\ensuremath{{\mathbb{Z}}}}
  
    \newcommand{\abar}{\overline a}
       \newcommand{\mbar}{\overline m}      
\newcommand{\sslash}{/\mkern-6mu/}

\title{Contractibility of Outer space: reprise}
\author{Karen Vogtmann}

\begin{document}
\begin{abstract} This note contains a newly streamlined version of the original proof that Outer space is contractible.  
\end{abstract}
\maketitle

\section{Introduction}

 In a series of lectures in August 2014 at the Seventh Seasonal Institute of the Mathematical Society of Japan  I began by recalling the construction of Outer space for a free group  $F_n$. This is a finite-dimensional contractible space with a proper action of the group $Out(F_n)$ of outer automorphisms of $F_n$ \cite{CV}.   I  then discussed the idea of developing an analogous outer space for the outer automorphism group of a general right-angled Artin group (RAAG) $A_\Gamma$.  Such a space was introduced in \cite{CSV} for the case that $A_\Gamma$ has no twist automorphisms.  We also understand  the case that $A_\Gamma$ is generated entirely by twist automorphisms and signed permutations, where the relevant space is a contractible subspace of the symmetric space for $SL(n,\mathbb R)$.  Our candidate for an Outer space for a general RAAG  is a hybrid of these two spaces.

The key result of \cite{CSV} is that the outer space constructed there is contractible. Although there are now several proofs that the original Outer space is contractible, the combinatorial techniques used in the original proof \cite{CV} are ultimately what  worked for us in  the more general RAAG setting. Thus the proof in \cite{CSV}   follows the original proof but  also incorporates simplifications, including some which were already introduced in    \cite{KV}.  Unfortunately, new complications  also arise due to the fact that the outer automorphism group of a general RAAG is  more complicated than the outer automorphism group of a free group.  
In this  note I will avoid these complications by just giving the complete simplified argument in the case of a free group. One motivation for doing this is to clarify the original proof, another is to  
make it easier for those who want to understand the general RAAG case to follow the argument.

\section{Outer space and its spine}  In this section we very briefly recall the definition of Outer space for a free group $F_n$ and its spine $K_n$.  For a more detailed introduction to these spaces, see \cite{V}.  

A {\em rose} is a graph with one vertex and $n$ edges (so the edges are loops at the vertex). We begin by fixing a specific rose $R_n$ whose fundamental group we identify with $F_n$.  
A point in Outer space is then an equivalence class  of  {\em marked metric graphs} $(G,g)$ (see Figure~\ref{mgraph}),  i.e.
\begin{enumerate}
\item $G$ is a finite metric graph with all vertices of valence at least $3$.
\item The volume of $G$  (i.e. the sum of the lengths of its edges) is $1$.
\item $g\colon R_n\to G$ is a homotopy equivalence.
\item $(G,g)$ is equivalent to $(G',g')$ if there is an isometry $h\colon G\to G'$ with $h\circ g\simeq g'$.
\end{enumerate}

\begin{figure}
\begin{center}
\begin{tikzpicture} 
\draw [thick, red, ->] (0,0) to [out=135, in=180]  (0,1);
\draw [thick, red] (0,1) to [out=0, in=60] (0,0);
\draw [thick, red, ->] (0,0) to [out=0, in=45]  (.65,-.65);
\draw [thick, red] (.65,-.65) to [out=225, in=-90] (0,0);
\draw [thick, red, ->] (0,0) to [out=270, in=-45]  (-.65,-.65);
\draw [thick, red] (-.65,-.65) to [out=-225, in=180] (0,0);
\node   (c) at (0,0) {$\bullet$};
\node (R) at (0,-1.2) {$R_3$};
\draw [thick, ->] (1.2,0) to (2.2,0);
\node (mu) at (1.7,.2) {$g$};
\node (t) at (3.5,.55) {$\bullet$};
\node (b) at (3.5,-.75) {$\bullet$};
\draw [thick]  (3.5, .85) circle (.3);
\draw [thick] (3.5,.55) to   (3.5,-.75);
\draw [thick] (3.5,.55) to  [out=0, in=0]  (3.5,-.75);
\draw [thick] (3.5,.55) to [out=180, in=180]  (3.5,-.75);
 
\node (G) at (3.5,-1.2) {$G$};
\end{tikzpicture}
\caption{A marked graph $(G,g)$}\label{mgraph}
\end{center}
\end{figure}
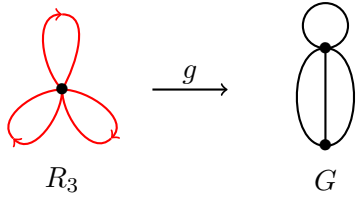
In the rest of this note  we will not be careful about distinguishing between a marked graph and its equivalence class.  
\begin{remark}  Requiring that the volume be equal to one is a means of normalizing  the projective class of a metric graph.  It is also sometimes convenient to consider other normalizations, or even to consider the  unprojectivized version of Outer space, where the edges of $G$ are allowed to have any positive lengths.
\end{remark}

Outer space is a union of open simplices, where the simplex containing $(G,g)$ consists of all marked graphs one can obtain  by varying the (positive) edge-lengths of $G$ while keeping the volume  equal to one.  Passing to a face of the simplex corresponds to shrinking some edges to points.    Some faces of each simplex are missing, since if an entire loop is shrunk to a point   the fundamental group is no longer $F_n$ and the induced marking is no longer a homotopy equivalence.  Formally including these missing faces gives a simplicial complex, called the {\em simplicial closure} of Outer space; the simplices which are not in Outer space are said to be {\em at infinity}.    

\begin{remark} The simplicial closure of Outer space is also called the {\em free splitting complex} or the {\em sphere complex}; these terminologies arise from different (equivalent) descriptions of Outer space, the first as a space of actions of $F_n$ on metric simplicial trees and the second as a space of weighted sphere systems in a doubled handlebody.  
\end{remark}

The  set of all open simplices of Outer space is partially ordered by the face relation, and the geometric realization of this partially ordered set (poset) is called the {\em spine} of Outer space.  Thus a simplex in the spine $K_n$ is a chain of open simplices $\sigma_0\subset \ldots \subset \sigma_k$ with $\sigma_i$ a proper face of $\sigma_{i+1}$.  In other words,  $K_n$  is a subcomplex of the barycentric subdivision of the simplicial closure.   There is a natural equivariant deformation retract of all of Outer space onto $K_n$, performed by pushing linearly from the (missing) simplices at infinity onto the spine.  

{\em Thus to prove that Outer space is contractible it suffices to show that the spine $K_n$ is contractible. } 

\section{Structure of the spine $K_n$ and plan of attack} 

Since vertices of $K_n$ sit at the barycenters of simplices of  Outer space all edges have the same length and we may think of these as purely combinatorial (as opposed to metric) objects or, equivalently,  assume all edges have length one.  We take this point of view for the rest of the paper.  

To describe the simplices of $K_n$, recall that a {\em forest}  in a graph $G$  is a subgraph  which contains no loops, i.e. a forest is a disjoint union of trees. Collapsing each tree of a forest $F$ to a point gives a new graph $G\sslash F$  and the collapsing map   $c_F\colon G\to G\sslash F$ is a homotopy equivalence, so   the composition $c_F\circ g$ is a marking of $G\sslash F$.  Forest collapse gives the vertices of $K_n$ the structure of a partially ordered set (poset).  The entire complex $K_n$ is the {\em geometric realisation}  (also called the {\em order complex}) of this poset.  In other words, there is an edge in $K_n$ from $(G,g)$ to $(G',g')$ whenever $(G',g')$ can be obtained from $(G,g)$ by a forest collapse, and there is a $k$-simplex for every chain of $k$ forest collapses \[(G_0,g_0)\to (G_1,g_1)\to \cdots\to (G_k,g_k).\]   

Every vertex in $K_n$ is connected by an edge to at least one marked rose $(R,r)$, obtained from $(G,g)$ by collapsing a maximal tree.  Thus $K_n$ is the union of the simplicial stars of its marked roses.
In order to prove that $K_n$ is contractible the idea is to build $K_n$ by starting with the star of a single marked rose (which is contractible), then attach the rest of the stars in some order and prove that at each stage we are attaching along something contractible.   In order to carry out this plan we need to
\begin{enumerate}
\item Define a {\em norm} on marked roses and prove that this norm well-orders the marked roses.
\item Identify which marked graphs in the star of a marked rose are {\em reductive}, i.e. are adjacent to marked roses of smaller norm.
\item  Prove that the subcomplex of reductive marked graphs in a star is contractible.  
\end{enumerate}
We actually perform a  bit of sleight-of-hand because it is easy to show that $K_n$ is connected.  We then show that {\em if\,} the subcomplex of reductive marked graphs is non-empty, {\em then} it is contractible.   This shows that $K_n$ is a union of contractible components, but since it is connected it is actually contractible.

\section{Connectivity of $K_n$ via Stallings folds}

If $h\colon R_n\to R$ is a homeomorphism then the marked rose $\rho_0=(R,h)$ is called the {\em standard rose}.  Since every vertex of $K_n$ is connected to a marked rose, to prove that $K_n$ is connected it suffices to connect any  marked rose to $\rho_0$.    The fact that you can do this follows easily from Nielsen's theorem that $Out(F_n)$ is generated by signed permutations and transvections (i.e. automorphisms which multiply one generator by another).  But there is also   a very slick,  completely elementary way to see this, due to Stallings (which also reproves Nielsen's theorem).  

A map $g\colon G\to H$ between two  graphs is called a {\em graph morphism} if it  is   sends vertices to vertices and  edges either to single edges or to vertices.  If we are allowed to add bivalent vertices to $G$ then any continuous map is homotopic to a graph morphism, so we will assume all of our maps are  graph morphisms.    

If a graph morphism  $g\colon G\to H$ is not locally injective then either some edge of $G$ must be mapped to a point in $H$ or two edges of $G$ emanating from the same vertex must map to the same edge in  $H$.  In either case, $g$ factors through a morphism $G\to G_1$, which either collapses an edge (in the first case) or folds two edges together (in the second case).  In the second case this morphism is called a {\em Stallings fold}.  An example is illustrated by the top arrow in Figure~\ref{folds}.    
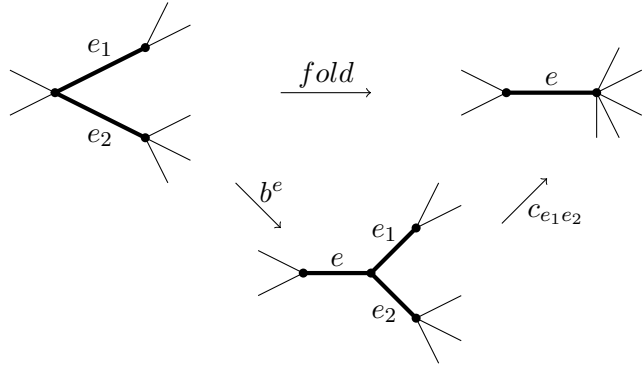
\begin{figure}
\begin{tikzpicture} [scale=.6]
\draw [line width=1.6pt] (2,1) to (0,0) to (2,-1);
\draw (-1,.5) to (0,0) to (-1,-.5);
\draw (2.5,2) to (2,1) to (3,1.5);
\draw (2.5,-2) to (2,-1) to (3,-1.5);
\draw (3,-.5) to (2,-1);
\fill (0,0) circle (.1);
\fill (2,1) circle (.1); 
\fill (2,-1) circle (.1); 
\node (eone) at (1,1) {$e_1$};
\node (etwo) at (1,-1) {$e_2$};
\node (phi)   at (6,.4) {$fold$};
\draw [->] (5,0) to (7,0);
\begin{scope}[xshift=10cm]
\draw (-1,.5) to (0,0) to (-1,-.5);
\draw [line width=1.6pt]  (0,0) to (2,0);
\draw (2.5,1) to (2,0) to (3,.5);
\draw (2.5,-1) to (2,-0) to (3,-.5);
\draw (2,0) to (2,-1);
\fill (0,0) circle (.1);
\fill (2,0) circle (.1); 
\node (e) at (1,.3) {$e$};
\end{scope}
\begin{scope}[xshift=6cm, yshift=-4cm]
\draw [line width=1.6pt] (2,1) to (1,0) to (2,-1);
\draw (-1.5,.5) to (-.5,0) to (-1.5,-.5);
\draw [line width=1.6pt] (-.5,0) to (1,0);
\draw (2.5,2) to (2,1) to (3,1.5);
\draw (2.5,-2) to (2,-1) to (3,-1.5);
\draw (3,-.5) to (2,-1);
\fill (-.5,0) circle (.1);
\fill (2,1) circle (.1); 
\fill (2,-1) circle (.1); 
 \fill (1,0) circle (.1);  
\node (eone) at (1.3,.9) {$e_1$};
\node (etwo) at (1.3,-.9) {$e_2$};
\node (e) at (.25,.3) {$e$};
\draw [->] (3.9,1.1) to (4.9,2.1);
\node (b)   at (-1.2,1.8) {$b^e$};
\draw [->] (-2,2) to (-1,1);
\node (cef)   at (5.1,1.3) {$c_{e_1e_2}$};
\end{scope}
\end{tikzpicture}
\caption{A fold is a blowup followed by a forest collapse}\label{folds}
\end{figure}

\begin{proposition} $K_n$ is connected.
\end{proposition}

\begin{proof} As remarked above, it suffices to connect any marked rose $\rho=(R,r)$ to the standard rose $\rho_0$.  We  begin by representing a homotopy inverse to $r$  by a graph morphism $s\colon~R~\to~R_n$  To do this we need to subdivide the edges of $R$ suitably.  Technically this is not allowed in $K_n$ since it introduces bivalent vertices, but we can recover the point of $K_n$ by simply ignoring the bivalent vertices.     
If $s$ is not locally injective, then either some edge collapses or you can fold two edges which start at the same vertex.  Note that these edges must have distinct terminal vertices, since otherwise they would form a loop with null-homotopic image, which can't happen because $s$ is a homotopy equivalence.

Recall that collapsing a forest in a marked graph gives an edge in  $K_n$ (unless the forest contains only proper subsets of subdivided edges, in which case collapsing does not change the point of $K_n$).   The reverse of a single edge collapse is called a {\em blowup}.  A fold corresponds to a  blowup followed by a forest collapse when the folded edges have distinct terminal vertices; thus a Stallings fold gives a path in the $1$-skeleton of $K_n$. See Figure~\ref{folds} for the case that neither edge is a loop. There is a similar picture if one edge is a loop. 
 
If $s$ is not locally injective, perform a fold.  (If a univalent vertex is produced, also collapse the adjacent edge; such a fold and collapse does not change the point of $K_n$.)   
If the induced map is not locally injective then fold or collapse again, thus producing a path in $K_n$.  This process has to stop because each time  the total  number of edges in the graph decreases.  When it stops,  the induced map $f\colon G\to R_n$ is locally injective. 

We now claim that a locally injective map $f$ is actually a homeomorphism, so the path in $K_n$ has arrived at $(f^{-1},G)= \rho_0$.  
To see this, let $x_i$ be the $i$-th petal of $R_n$.  Since $f$ is a homotopy equivalence, there is some loop $\ell_i$ in $G$ with $f(\ell_i)\simeq x_i$.  Since $f$ is locally injective, $\ell_i$ is a simple loop in $G$.  If $i\neq j$, then the loops $\ell_i$ and $\ell_j$ can intersect in at most a point, since that is true of their images $x_i$ and $x_j$, so any overlap would have to collapse to a point.  The union of the $\ell_i$ must be all of $G$, since otherwise the complement would be a forest which must collapse.  Finally  the $\ell_i$ must all intersect in the same point, forming a rose. 
\end{proof}

 \section{The norm of a rose}
 The next task is to find a ``Morse function" which totally orders the roses. The idea is that more complicated markings should come later in the ordering.  
 
For any loop $\gamma$ in a graph $G$, let $\ell_G(\gamma)$ denote the length of the shortest loop in the (free) homotopy class of $\gamma$,  where we think of each edge of $G$ as having length one.  Note that $\gamma$ is shortest  in its homotopy class if and only if it is  locally injective, in which case we call  it   a {\it tight} loop.

Fix a basis $x_1,\ldots,x_n$ for $F_n$  and list the conjugacy classes in order of increasing (cyclically reduced)  word-length:
 \[\mathcal W=(w_1,w_2,\ldots)=(x_1,x_2,\ldots, x_n, x_1^2,\ldots,x_1x_2,\ldots, x_1x_2^{-1},\ldots,x_1^3,\ldots ).\]   Note that it is redundant to include both $w$ and $w^{-1}$ in $
 \mathcal W$, so we won't. Let $\Z^{\mathcal W}$ be the associated ordered abelian group, with the lexicographical order.  For any rose $\rho=(r,R)$  and any $w\in \mathcal W$, define $\|\rho\|_w$ to be equal to $\ell_R(r(w))$.  The norm of  $\rho$ is then defined by \[\|\rho\|=\|\rho\|_{\mathcal W} =  (\|\rho\|_{w_1}, \|\rho\|_{w_2} ,\ldots)\in \Z^{\mathcal W}.\] 
This norm totally orders the roses, according to the following basic theorem, proved independently by  Alperin and Bass and by Culler and Morgan:

\begin{theorem}\label{Chiswell} \cite{AB, CM} A free action of $F_n$ on a simplicial tree is determined by its translation length function.
\end{theorem}

We will show that our norm has the following stronger property, which we will need to do induction:

\begin{proposition}\label{ordering} The set of  roses is well-ordered in this norm.
\end{proposition}

\begin{lemma}\label{finite} Any free minimal action of $F_n$ on a simplicial tree is determined by the translation lengths of finitely many conjugacy classes.  
\end{lemma}

\begin{proof} Assume first that the quotient by the action is a rose.  We claim that in this case the action is determined by the translation lengths of the $2n^2+2n$ classes which have length at most $2$.  These are represented by oriented loops of length at most $2$ in the quotient rose.  The petals of the rose correspond to some basis $w_1,\ldots,w_n$ for $F_n$.     

Suppose there was another free minimal action with quotient a rose in which these classes have the same lengths.  This action is the first action twisted by an (outer) automorphism $\hat\phi$.   
Any automorphism $\phi$ representing $\hat \phi$ must permute  the conjugacy classes of $w_1,w_1^{-1},\ldots,w_n,w_n^{-1}$ since these are all of the classes of length $1$.   
In fact, we may assume $\phi$ sends each $w_i$ to a conjugate of itself, since permuting and inverting the $w_i$ can be realized by  an isometry of the rose, which lifts to  an equivariant isometry of the tree.    

Take a representative $\phi$ for $\hat\phi$ with  $\phi(w_1)= w_1$, and suppose $\phi(w_2)=uw_2u^{-1}.$  Then $\phi(w_1w_2)=w_1uw_2u^{-1}$.  Since this is conjugate to an element of length $2$, $u$ must be a power of $w_1$.   Therefore, after composing $\phi$ with conjugation by $u^{-1}$  we may assume $\phi(w_1)=w_1$ and $\phi(w_2)=w_2$.  

Now consider $\phi(w_i)=vw_iv^{-1}$ for  $i>2$.  The argument above shows that $v$ must be a power of $w_1$ {\em and} a power of $w_2$, so in fact $v=1$ and $\phi$ is the identity. 

If the quotient by the action is a marked graph which is not a rose, choose a maximal tree and collapse it to get a rose.  We can distinguish this rose from any other rose by the lengths of finitely many conjugacy classes.  Our original marked graph is obtained from this rose by blowing up the vertex into a tree.  There are only finitely many ways to do this, which by Theorem~\ref{Chiswell} can be distinguished by finitely many lengths.  Thus the entire action is determined by the lengths of finitely many conjugacy classes.  
\end{proof}

\begin{remark}  The finite set of conjugacy classes found in Lemma~\ref{finite} depends on the action.  It can be shown that for 
 any fixed finite set of conjugacy classes there are two roses in which those conjugacy classes have the same translation length. (In fact, you can find an arbitrary (finite) number of roses in which those conjugacy classes have the same translation length, see \cite{SV}). 
\end{remark}

\begin{proof}[Proof of Proposition~\ref{ordering}]
We have to show that any subset $U$ of roses has a least element.  
Set $U=U_0$ and define a decreasing chain  
\[ U=U_0\supset U_1\supset U_2\ldots, \]
where $U_i$ is the set of elements in $U_{i-1}$ for which $\|\rho\|_{w_i}$ is minimal possible, say $\|\rho\|_{w_i}= \ell_i$.  Note that   each element of $U_i$ is $\leq$ each element of $U_{i-1}$ in the ordering.  

The function $f\colon F_n\to  \Z$ defined by $f(w_i)=\ell_i$ satisfies the axioms for a translation length function, so corresponds to an action on a tree.  This action is free since    $f(w_i)=\ell_i\neq 0$ for all $i$.  Therefore it corresponds to a marked graph $\gamma$ (with all edges of length 1). By Lemma~\ref{finite} 
 $\gamma$ is determined by the lengths of finitely many elements of $F_n$.  So for $N$ sufficiently large $U_N$ has only one element, a rose which must be equal to $\gamma$, which is smaller than any other element of $U$.  
\end{proof}

\section{Separating edges}

It is convenient to prove contractibility just for the subcomplex $L_n$ of $K_n$ spanned by graphs with no separating edges.  This is justified by the following observation.  

\begin{proposition}\label{nonsep} $K_n$ deformation retracts to the subcomplex $L_n$ spanned by graphs with no separating edges.
\end{proposition}

The deformation retraction is easy to see:  one just uniformly shrinks all separating edges to zero.  Since $K_n$ is  the geometric realisation of a poset, one can give a formal proof  using  Quillen's Poset Lemma, which will also come in handy later.

\begin{lemma}\label{Quillen}[Quillen's Poset Lemma \cite{Q}] Let $P$ be a poset and $f\colon P\to P$   a poset map (i.e. $x\leq y$ implies $f(x)\leq f(y)$).  If in addition $f(x)\leq x$ for all $x$, then $|P|$ deformation retracts to $|f(P)|$, where vertical bars denote geometric realization. 
\end{lemma}
Note that by using the opposite poset we can draw  the same conclusion if $f(x)\geq x$ for all $x$.  The proof of the Poset Lemma is  a straightforward application of the prism operator, and is left to the reader.

\begin{proof}[Proof of Proposition~\ref{nonsep}] The map $f\colon K_n\to K_n$ which contracts each separating edge is a poset map with image $L_n$.  
\end{proof}

We reiterate our plan of attack, as an excuse to introduce some notation.
All roses are in $L_n$, and we view $L_n$ as the union of the simplicial stars of its roses:  \[ L_n=\bigcup_{\hbox{\tiny roses }\rho} st(\rho). \]
We construct $L_n$ by starting with the star of the (unique) rose  of minimal norm and adding stars of roses in the order dictated by the norm, i.e. 
for each rose $\rho$, define  \[ L_{<\rho}=\bigcup_{\|\rho^\prime\|<\|\rho\|} st(\rho^\prime). \]
We will prove that if $st(\rho)\cap L_{<\rho}$ is non-empty, then it is contractible.  This will show that $L_n$ is a union of contractible components.  But we already know $L_n$ is connected, so it is contractible. 

\section{Reductive graphs and the Factorization lemma}

We call a marked graph {\em reductive} if it is in $st(\rho)\cap L_{<\rho}$; thus    $(G,g)$ is  reductive if and only if $G$ contains  maximal trees $\Phi$ and $F$ such that collapsing $\Phi$ gives $\rho$ and collapsing $F$ gives a different rose $\rho^\prime$ with $\|\rho^\prime\|<\|\rho\|$.    

For  each edge $e\in G$ of  a marked graph $(G,g)$, let  $|e|\in \Z^{\mathcal W}$ be the element whose coefficient $|e|_w$ is the number of times a tight representative of $g(w)$ crosses $e$, in either direction.  Since collapsing a forest sends tight paths to tight paths, we get \[ \|\rho^\prime\|=\|\rho\|+\sum_{\alpha\in\Phi}|\alpha|-\sum_{e\in F}|e|. \]

\begin{proposition}\label{switch}  If $\Phi=\{\alpha_1,\ldots,\alpha_k\}$ and $F=\{e_1,\ldots,e_k\}$ are maximal trees in a graph $G$, then there is a permutation $\sigma$ of $\{1,\ldots,n\}$   such that 
$e_{\sigma(i)}$ connects the two components of $\Phi-\alpha_i$ for each $i$.
\end{proposition}

\begin{proof}   First of all we want $\sigma$ to be the identity on common edges of $\Phi$ and $F$.  We can then  contract these edges to reduce the problem to the case that   $\Phi$ and $F$ have no edges in common.  

It's easy to move from $\Phi$ to $F$ by replacing one edge at a time, but we want to do something a little more subtle than that.  Here's one proof you can do it:  

$\Phi\cup F$ is a graph of rank $k$, and the edges of $\Phi$ or of $F$ each give a basis for $H_1(\Phi\cup F)$.  The change of basis matrix $B$ has entry $b_{ij}=\pm 1$  if the unique path in $F$ connecting the endpoints of $\alpha_j$ crosses $e_i$; otherwise $b_{ij}=0$.  Since this matrix is non-singular, some term 
\[ sign(\sigma)\prod_{i=1}^n  b_{i,\sigma(i)} \] in the expression for the determinant of $B$ is non-zero. Then $e_{\sigma(i)}$ joins the two components of $\Phi-\alpha_i$.
\end{proof}

\begin{corollary}[Factorization lemma]  In $st(\rho)$, every reductive $(G,g)$ is adjacent to a 2-vertex reductive graph. 
\end{corollary}

\begin{proof}  With $\Phi$, $F$ and $\sigma$ as in Lemma~\ref{switch}, we have \[ \|\rho^\prime\|=\|\rho\|+\sum_{i}|\alpha_i|-\sum_{i }|e_{\sigma(i)}|=\|\rho\|+\sum_i(|\alpha_i| -|e_{\sigma(i)}|). \]
Since $\|\rho^\prime\|<\|\rho\|$, we must have $|\alpha_i| -|e_{\sigma(i)}|<0$ for some $i$.  Then the two-vertex graph obtained by collapsing all edges of $\Phi-\alpha_i$ is reductive (and is connected to $(G,g)$). 

\end{proof}

Thus we may view $st(\rho)\cap L_<\rho$ as a union of stars of 2-vertex graphs.  If we are lucky there is only one reductive 2-vertex graph  so $st(\rho)\cap L_<\rho$  is contractible.  We are seldom so lucky, however, and need to work harder.  In order to do this we will reinterpret reductive graphs and the norm using a neat combinatorial model originally introduced by Whitehead in the context of sphere complexes in doubled handlebodies.  This model (translated into the language of partitions and graphs instead of sphere complexes) is explained in the next section.  

\section{Ideal edges}

We now fix a rose $\rho=(r,R)$ and re-interpret graphs in $st(\rho)$ in terms of partitions of the set $H$ of half-edges of $R$.  We denote the natural involution on $H$ by $e\mapsto \overline e$.  This section explains the translation from graphs to partitions.

A marked graph $(G,g)$  is in $st(\rho)$ if and only if $G$ has a maximal tree $\Phi$ so that the composition of the collapsing map $c_\Phi$ with $g$ is homotopic to $r$.  In particular, the edges of $G - \Phi$ are mapped homeomorphically onto the edges of $R$, so if you snip each edge of $G-\Phi$ you obtain a tree whose leaves are labelled by the elements of $H$ (see Figure~\ref{snip}). Each edge $\alpha$ of $\Phi$ gives a partition of $H$ into two subsets, called the {\em sides} of $\alpha$, and different edges $\alpha$ and $\beta$ give {\em compatible} partitions, in the sense that $A\cap B=\emptyset$ for some choice of sides $A$ of $\alpha$ and $B$ of $\beta$.    Conversely, given any set of compatible partitions of $H$ we can reconstruct the tree $\Phi$ with leaves labeled by $H$, and recover $(G,g)$ by reconnecting the paired elements of $H$.

\begin{figure}\begin{center}
\begin{tikzpicture} [thick, scale=.45]
\begin{scope}[decoration={markings,mark = at position 0.5 with {\arrow{stealth}}}] 
\draw [black!40!green, line width=2pt] (0,0) to (2,3) to (4,0);
\midarrow (0,0) to [out=120, in=180] (2,3);
\midarrow  (0,0) to [out=-30, in=210] (4,0);
\draw [->]  (2,3) to [out=150, in=180] (2,5);
\draw (2,5) to [out=0, in=30] (2,3);
\draw [->] (4,0) to  [out=60, in=90] (6,0);
\draw (6,0) to  [out=270, in=300] (4,0);
\node (e1) at (-.5,2) {$e_1$};
\node (e2) at (2,0) {$e_2$};
\node (e3) at (5,1) {$e_3$};
\node (e4) at (3.2,4) {$e_4$};
\end{scope}

\begin{scope}[xshift = 11cm]
\draw [black!40!green, line width=2pt] (0,0) to (2,3) to (4,0);
\draw (-.75,.75) to (.75,-.75);
\draw (3.25 ,-.8) to (4,0) to (5,.75); 
\draw (4,0) to (5,-.75); 
\draw (.75,3) to (3.25,3); 
\draw (2,3) to (2,4.25); 
\node (e1bar) at (-1.1,1.1) {$\overline e_1$};
\node (e1) at (.25,3) {$ {e_1}$};
\node (e1bar) at (1.1,-1.1) {$\overline e_2 $};
\node (e2) at (3,-1.15) {${e_2}$};
\node (e3) at (5.5,1) {$e_3$};
\node (e3bar) at (5.5,-1.15) {$\overline e_3$};
\node (e4) at (2,4.75) {$e_4$};
\node (e4bar) at (3.75,3 ) {$\overline e_4 $};
\end{scope}

\begin{scope}[xshift = 22cm, yshift=.5cm]
\node (e1bar) at (-1,1) {$\overline e_1$};
\node (e1) at (.5,3) {$ {e_1}$};
\node (e1bar) at (1,-1) {$\overline e_2 $};
\node (e2) at (3.2,-1.05) {${e_2}$};
\node (e3) at (5.3,.8) {$e_3$};
\node (e3bar) at (5.3,-1.05) {$\overline e_3$};
\node (e4) at (2,4.5) {$e_4$};
\node (e4bar) at (3.5,3 ) {$\overline e_4 $};
\draw[black!40!green, densely dotted, rotate=-45, line width = 0.5mm] (0, 0) ellipse (2.5cm and 1.5cm);
 \draw[black!40!green, densely dotted, rotate=45, line width = 0.5mm] (3, -3.5) ellipse (2.5cm and 2cm);
\end{scope}
\end{tikzpicture}
\caption{A graph in $st(\rho)$ is equivalent to a partition of the half-edges of $\rho$}\label{snip}
\end{center}
 \end{figure}
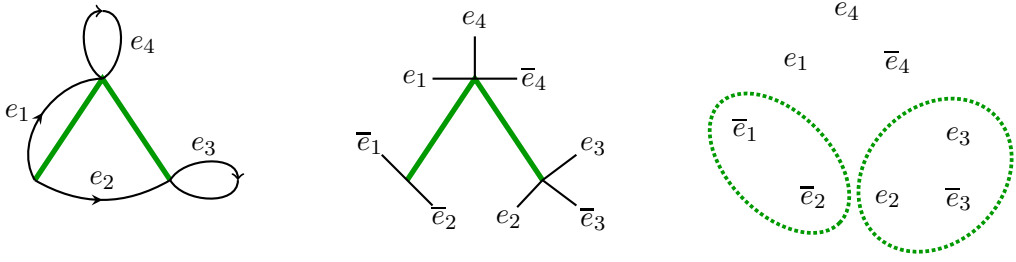

\begin{definition}  A partition of the set $H$ of half-edges of $R$ is an {\em ideal edge} if it separates some pair $\{e, \overline e\}$.   An ideal edge is {\em trivial} if one side is a singleton.  An {\em ideal tree} is a  set of non-trivial ideal edges which are pairwise compatible.   
\end{definition}

Note that a trivial ideal edge corresponds to a graph with a bivalent vertex, and a partition which is not an ideal edge corresponds to a graph with a separating edge.  Since none of our graphs have bivalent vertices or separating edges, the vertices of $st(\rho)$ correspond  to ideal trees.  Collapsing an edge of a tree corresponds to removing a partition from the associated ideal tree.  Thus the simplicial complex $st(\rho)$ is the geometric realization of the poset of ideal trees, ordered by inclusion.  

We say an ideal tree or edge is {\it reductive} if the corresponding graph is reductive. 
The Factorization Lemma, reinterpreted in this language, says that if $\Phi$ is a reductive ideal tree, then it contains a reductive ideal edge.  

If $\alpha$ is a reductive ideal edge, then there is some pair $\{e,\bar e\}$ in $H$ separated by $\alpha$ with $|\alpha|-|e| <0$.  If $A$ is the side of $\alpha$ containing $e$, we call $(A, e)$ a {\em reductive pair} for $\alpha$.   

\section{The star graph and the norm}\label{star}

We have reinterpreted graphs in $st(\rho)$ as partitions of the half-edges of $\rho$.  We now need to interpret the norm of a rose in this model.  In order to understand which partitions are in the reductive subcomplex of $st(\rho)$, we also need to interpret $|e|$ and $|\alpha|$ in this model.  

Fix a rose $\rho=(r,R)$ and let $H$ be the half-edges of $R$.  For any conjugacy class $w$ we associate a graph $\Gamma_w$ called the {\em star graph} of $w$.  The vertices of $\Gamma_w$ are the elements of $H$. To define the edges take a tight loop representing $r(w)$ in $R$, then snip the edges of $R$ to make a tree with 2n leaves.  This cuts $r(w)$ into segments joining the cuts; these are the edges of $\Gamma_w$ (Figure~\ref{star graph}). Since $\|\rho\|_w$ is the length of $r(w)$, the sum of the valences of $\Gamma_w$ is twice $\|\rho\|_w$.

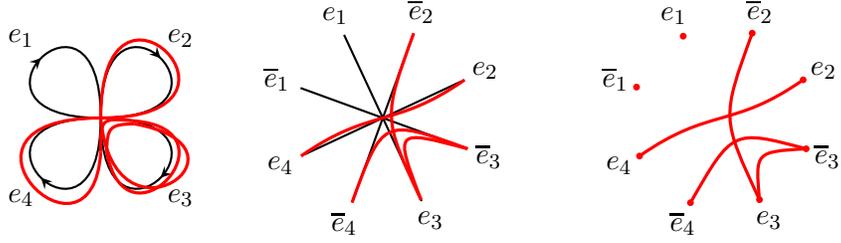
\begin{figure}
\begin{center}
\begin{tikzpicture}[thick,  scale=1.5]  
 \begin{scope}[rotate=45, decoration={markings,mark = at position 0.5 with {\arrow{stealth}}}]
\midarrow (0,0) .. controls (-1,1) and (1,1)   .. (0,0);
\midarrow (0,0) .. controls (1,1) and (1,-1)   .. (0,0);
\midarrow (0,0) .. controls (1,-1) and (-1,-1)   .. (0,0);
 \midarrow (0,0) .. controls (-1,-1) and (-1,1)   .. (0,0);
\node (e1) at (90:1) {$ {e_1}$};
\node (e2) at (0:1) {$e_2$};
\node (e3) at (270:1) {$e_3$};
\node (e4) at (180:1) {${e_4}$};
\draw [red, line width=1.2pt] (0,0)  .. controls (1.1,1.1) and (1.1,-1.1)   .. (0,0); 
\draw [red, line width=1.2pt] (0,0)  .. controls (-1.1,1.1) and (-1.1,-1)   .. (225:.3);
\draw [red, line width=1.2pt] (225:.3)  .. controls (0,0)   .. (-50:.3); 
 \draw [red, line width=1.2pt] (-50:.3)  .. controls (-50:1.25) and (235:1.25)  .. (235:.3);
  \draw [red, line width=1.2pt] (235:.3)  .. controls (260:.1) and (280:.1)   .. (305:.3) ;
   \draw [red, line width=1.2pt] (305:.3)  .. controls (-55:1.5) and (230:1.25)  .. (0,0);
\end{scope}
\begin{scope}[xshift = 2.5cm]
\node (e2bar) at (70:1) {$ \overline e_2$};
\node (e1) at (115:1) {$e_1$};
 \node (e1bar) at (160:1) {$\overline e_1 $};
 \node (e4) at (205:1) {$e_4 $};
 \node (e4bar) at (250:1) {$\overline e_4$};
 \node (e3) at (295:1) {$ e_3$};
 \node (e3bar) at (-20:1) {$\overline e_3$};
 \node (e2) at (25:1 ) {$e_2$};
 \draw (e1) to (e3);
  \draw (e2) to (e4);
   \draw (e1bar) to (e3bar);
    \draw (e2bar) to (e4bar);
\draw [red, line width=1.2pt] (295:.8) .. controls (.1,-.1)     .. (-20:.8);  
  \draw [red, line width=1.2pt] (250:.8) .. controls (0,0)     .. (-20:.8);
   \draw [red, line width=1.2pt] (205:.8) .. controls (-.1,.1) and (.1,-.1 ) .. (25:.8);
       \draw [red, line width=1.2pt] (295:.8) .. controls (0,0)   .. (70:.8);
\end{scope}
\begin{scope}[xshift = 5.5cm]
\node (e2bar) at (70:1) {$ \overline e_2$};
\node (e1) at (115:1) {$e_1$};
 \node (e1bar) at (160:1) {$\overline e_1 $};
 \node (e4) at (205:1) {$e_4 $};
 \node (e4bar) at (250:1) {$\overline e_4$};
 \node (e3) at (295:1) {$ e_3$};
 \node (e3bar) at (-20:1) {$\overline e_3$};
 \node (e2) at (25:1 ) {$e_2$};
 \fill [red, line width=1.2pt] (70:.8) circle (.03);
  \fill [red, line width=1.2pt] (115:.8) circle (.03);
   \fill [red, line width=1.2pt] (160:.8) circle (.03);
    \fill [red, line width=1.2pt] (205:.8) circle (.03);
            \fill [red, line width=1.2pt] (250:.8) circle (.03);
     \fill [red, line width=1.2pt] (295:.8) circle (.03);
      \fill [red, line width=1.2pt] (-20:.8) circle (.03);
       \fill [red, line width=1.2pt] (25:.8) circle (.03);       
 \draw [red, line width=1.2pt] (295:.8) .. controls (.3,-.3)     .. (-20:.8); 
  \draw [red, line width=1.2pt] (250:.8) .. controls (.1,-.1)     .. (-20:.8);
    \draw [red, line width=1.2pt] (205:.8) .. controls (-.1,.1) and (.1,-.1 ) .. (25:.8);
        \draw [red, line width=1.2pt] (295:.8) .. controls (0,0)   .. (70:.8);
\end{scope}
\end{tikzpicture}
\caption{Star graph of the loop $w=e_2e_4^{-1}e_3^2$}\label{star graph}
\end{center}
\end{figure}

Let $(G,g)$ be the two-vertex graph represented by a single ideal edge $\alpha$.  Recall that $|e|_w$ denotes the number of times a tight loop representing $g(w)$ passes over $e$.  If $e\neq \alpha$, then we could also measure $|e|_w$ in $\rho$ where   $|e|_w$ is equal to the valence of $e$ (or of $\overline e$) in the star graph $\Gamma_w$.  If $e=\alpha$, then $|\alpha|_w$ is equal to the number of edges of $\Gamma_w$ with one vertex on each side of $\alpha$.

\section{Contractibility of $st(\rho)\cap L_{<\rho}$}

Since we have reinterpreted $st(\rho)$ as the geometric realization of the poset of   ideal trees, we will make free and frequent use of Quillen's Poset Lemma (Lemma~\ref{Quillen}).

The complex we are interested in, $st(\rho)\cap L_{<\rho},$ is the poset  of {\em reductive} ideal trees; call it $P$.   As we remarked in section~\ref{star}, the Factorization Lemma says that each $\Phi\in P$ contains at least one reductive ideal edge.  Therefore the map $P\to P$ which throws out all of the {\it non}-reductive ideal edges is a well-defined map.  It satisfies the hypotheses of the Poset Lemma, so $|P|$ deformation retracts to its image, which is the realization of the subposet $Q$ of {\it strictly reductive} ideal trees, i.e. ideal trees all of whose edges are reductive.

Now choose a maximally reductive ideal edge $\mu$, i.e. there is a reductive pair $(M,m)$ for $\mu$ with $|m|-|\mu|$ maximal among all reductive pairs $(A,a)$.    
If every other reductive ideal edge is compatible with $\mu$, then $\mu$ can be added to every strictly reductive ideal tree $\Phi$, and the poset maps $\Phi\mapsto \Phi\cup\{\mu\}\mapsto \{\mu\}$ retract $|Q|$ to the single point $\mu$.

If there are edges $\alpha$ which cross $\mu$, we have to work harder.  Here is the Key Lemma (see Figure~\ref{keylemma}):

\begin{lemma}[Key Lemma] Let $\mu$ be a maximally reductive ideal edge, with maximal pair $(M,m)$,  let $\alpha$ be a reductive edge that crosses $\mu$, and let $A$ be the side of $\alpha$ that contains $m$.    Then  $A\cup M $ or $\overline A\cap M$ is one side of a reductive ideal edge $\gamma$.
\end{lemma}

Assuming the Key Lemma, we proceed as follows:

Choose $\alpha,$ with side $A$ containing $m$, such that  
\begin{itemize}
\item If $\beta$  is compatible with $\alpha$ and the side $B$ of $\beta$ containing $m$ also contains $A$,  then  $B\supset M$.  
\end{itemize}

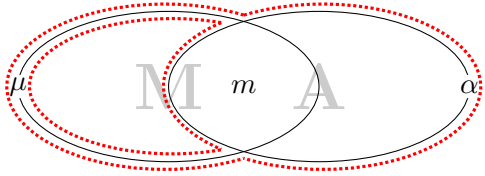
\begin{figure}
\begin{center}
\begin{tikzpicture} 
 \colorlet{lightgray}{black!20}     
    \node (A) at (0,.0) {\Huge\bf \textcolor{lightgray}{M}};
       \node (M) at (2,.0) {\Huge\bf \textcolor{lightgray}{A}}; 
          \draw (0,0) ellipse (2cm and 1cm);
             \draw (2,0) ellipse (2cm and 1cm);
  \draw [white, fill=white] (-2,0) circle (.15);
  \draw [white, fill=white] (4,0) circle (.15);
    \node (x) at (-2,0) {$\mu$};
    \node (m) at (1,0) {$m$};
      \node (z) at (4,0) {$\alpha$};
         \draw [red, line width=1.2pt, densely dotted] (1.0,.95) arc(60:300:2.1cm and 1.1cm);
          \draw [red, line width=1.2pt, densely dotted] (1.0,.95) arc(120:-120:2.1cm and 1.1cm);
  \draw [red, line width=1.2pt, densely dotted] (.7,.85) arc(70:290:1.9cm and .9cm);
          \draw [red, line width=1.2pt, densely dotted] (.7,.85) arc(130:230:2.15cm and 1.1cm);
 \end{tikzpicture}
 \caption{Key Lemma:  One of the dotted ideal edges is reductive}\label{keylemma}
 \end{center}
 \end{figure}
 
By the Key Lemma, one of $M\cap \overline A$ or $M\cup A$ determines a reductive ideal edge $\gamma$.  We now observe that $\gamma$  is compatible with $\alpha$, $\mu$ and with every $\beta$ compatible with $\alpha$.  Therefore 
\[
\Phi \mapsto  \begin{cases}
  \Phi\cup \gamma &\mbox{if } \alpha\in\Phi \\
   \Phi &\mbox{if }  \alpha\not\in\Phi
   \end{cases}
\]
is a poset map. It satisfies the condition of the Poset Lemma, so retracts $|Q|$ to its image. 

In the image, everything that contains $\alpha$ also contains $\gamma$.  Then the map throwing $\alpha$ out of every $\Phi$ that contains it is also a poset map.  The final effect is to replace the edge $\alpha$ which crosses $\mu$ by the edge $\gamma$ which is compatible with $\mu$, i.e. the image is now all reductive ideal trees which do not contain $\alpha$.  

We repeat this procedure until we have eliminated every ideal edge which crosses $\mu$.  Then we can retract to $\mu$ as before.

\section{Proof of the Key Lemma}

In this section we prove  the Key Lemma needed in the proof of contractibility.  Recall we are working with a fixed rose $\rho=(r,R)$ and partitions of the set of half edges $H$ of $R$.  We first define the {\it dot product} $A.B$ of disjoint subsets $A$ and $B$ of $H$ 
as the element of $\Z^\mathcal W$ with coordinate $(A.B)_w$ equal to the number of edges in the star graph $\Gamma_w$ with one vertex in $A$ and one vertex in $B$.  

For $A\subseteq H$, set  $\overline A=H\setminus A$  and  $|A|=A.\overline A$. As noted in section~\ref{star},  for $e\in H$ $|e|_w$ is just the valence of $e$ in the  star graph 
$\Gamma(w)$, and if $A$ is either side of an ideal edge $\alpha$, then $|A|_w=|\alpha|_w$. 
 
 We use ``+" to denote disjoint union of sets, as well as addition in  $\Z^\mathcal W$, resulting in the following pleasing formulas.

 \begin{lemma} If $A, B$ and $C$ are disjoint subsets of $H$, then
 \begin{enumerate} \item $A.B=B.A$
 \item $A.(B+C)=A.B + A.C$
 \end{enumerate}
 \end{lemma}
 \begin{proof} Straghtforward.
 \end{proof}

\begin{lemma}\label{count} Let $A$ and $B$ be subsets of $H$.  Then
$|A\cap B|+|A\cup B|\leq |A|+|B|.$
\end{lemma}

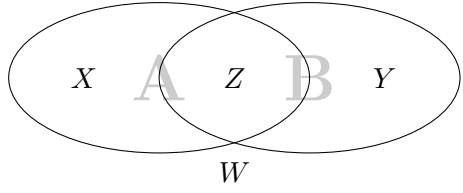
\begin{figure}
\begin{center}
\begin{tikzpicture} 
 \colorlet{lightgray}{black!20}
    \node (A) at (0,.0) {\Huge\bf \textcolor{lightgray}{A}};
       \node (B) at (2,.0) {\Huge\bf \textcolor{lightgray}{B}}; 
          \draw (0,0) ellipse (2cm and 1cm);
             \draw (2,0) ellipse (2cm and 1cm);
    \node (x) at (-1,0) {$X$};
    \node (y) at (1,0) {$Z$};
      \node (z) at (3,0) {$Y$};
        \node (2) at (1,-1.25) {$W$};
 \end{tikzpicture}
 \caption{Diagram for Lemma~\ref{count}}\label{countdiag}
\end{center}
\end{figure}

\begin{proof} $A$ and $B$ together partition $H$ into disjoint subsets $Z=A\cap B$,  $W=\overline A\cap \overline B$, $X=A\cap\overline B$ and $Y=B\cap\overline A$ (see Figure~\ref{countdiag}).  We  compute
\[ 
|A|=(X+Z).(Y+W)= X.Y +X.W + Z.Y + Z.W 
\]
\[ 
|B|=(Y+Z).(X+W) = Y.X + Y.W + Z.X + Z.W 
\]
and
\[ 
|A\cap B|= Z.(X+Y+W) = Z.X+Z.Y+Z.W 
\]
\[  
|A\cup B| = (X+Z+Y).W= X.W+Z.W+Y.W 
\] 
So altogether we have $|A|+|B|=|A\cap B|+|A\cup B|+2(X.Y);$  in particular
\[
|A\cap B|+|A\cup B|\leq |A|+|B|.
\]
\end{proof}

\begin{lemma}[Key Lemma] Let $\mu$ be a maximally reductive ideal edge, with maximal pair $(M,m)$,  let $\alpha$ be a reductive edge that crosses $\mu$, and let $A$ be the side of $\alpha$ that contains $m$.    Then  $A\cup M $ or $\overline A\cap M$ is one side of a reductive ideal edge $\gamma$.
\end{lemma}

\begin{proof}[Proof of Key Lemma] Since $\alpha$ and $\mu$ cross, together they  partition $H$ into four disjoint subsets, which we will call {\it sectors}. Since $\alpha$ is reductive, there is $a\in A$ with $|a|-|\alpha|>0$.    The proof falls into cases depending on the locations of $a, \abar$ and  $\mbar$.

If each sector contains one of $a,\abar, m,\mbar$,  then the only possibility is  $a\in A\cap \overline M, \abar\in M\cap\overline A$ and $\mbar\in \overline A\cap \overline M$ (since we already have $m\in A\cap M$).  In this case Lemma~\ref{count} gives us
 
\[
 |A\cap M|+|A\cup M|+|\overline A\cap M|+|\overline A\cup M|
  \leq 2|A|+2|M|
\]
 or
\[
(|m| -|A\cap M|)+(|m| -|A\cup M|)+(|a| -|\overline A\cap M|)+( |a| -|\overline A\cup M|) 
\]
\[\geq 2(|a|-|A|) +2(|m| -|M|) 
\]
so
\[
(|m| -|A\cup M|)+(|a| -|\overline A\cap M|)\geq 
\]
\[
 2(|a|-|A|) +\big[(|m| -|M|)-(|m| -|A\cap M|)\big] +\big[(|m| -|M|)-(|a| -|\overline A\cup M|)\big] 
\]
Since $(A,a)$ is a reductive pair and $(M,m)$ is maximally reductive, each of the three terms on the bottom line is positive, so the sum on the next line up is positive,  which implies that at least one of $(A\cup M,m)$ or $(\overline A\cap M,a)$ is a reductive pair, as required.

We may now assume some sector contains none of $a,\abar,m,\mbar$.    Since $a$ and $\abar$ (resp. $m$ and $\mbar$) can't be in the same sector, some sector has    $a$ or $\abar$ and $m$ or $\mbar$.  Replacing $(M,m)$ by $(\overline M, \mbar)$ if necessary, we may assume $a,m\in A\cap M$.  The rest of the proof breaks into cases depending on the positions of $\abar$ and $\mbar$.

If $\mbar \in \overline A\cap \overline M$, then  the inequality $|A\cap M|+|A\cup M|\leq |A|+|B|$ from Lemma~\ref{count} gives 
\[
(|a|-|A\cap M|)+(|m|-|A\cup M|)\geq (|a|-|A|)+(|m|-|M|)
\]
so  
\[
|m|-|A\cup M|\geq \big[|a|-|A|\big]+\big[(|m|-|M|)-(|a|-|A\cap M|)\big].
\]

  Since $(M,m)$ is maximally reductive, both terms on the right hand side are positive, showing that $(A\cup M, m)$ is a reductive pair.

If $\mbar \in A$ and $\abar \in \overline A\cap \overline M$, the same proof with the roles of $a$ and $m$ switched shows $(A\cup M,a)$ is a reductive pair.  The only remaining case is $\mbar\in A, \abar\in M$; in this case we use Lemma~\ref{count} with the sets $\overline A$ and $M$ to get
\[
|\overline A\cap M| + |\overline A\cup M|\leq |\overline A|+|M|=|A| + |M|
\] 
so
\[
|m|-|\overline A\cap M| + |a|-|\overline A\cup M|\geq |a|-|\overline A|+|m|-|M|
\] 
\[
|m|-|\overline A\cap M| \geq (|a|-|\overline A|)+\big[(|m|-|M|) - (|a|-|\overline A\cup M|)\big],
\]   
and again both terms on the right are positive since $(M,m)$ is maximally reductive, showing $(\overline A\cap M,m)$ is a reductive pair.

\end{proof}

\bibliographystyle{siam}

\end{document}